\RequirePackage{amsmath}
\documentclass[10pt, reqno]{amsart}
\usepackage{amssymb,url, color, mathrsfs}
\usepackage[colorlinks=true, bookmarks=true, pdfstartview=FitH, linktocpage=true, linkcolor = magenta, citecolor = blue]{hyperref}
\usepackage{longtable}
\usepackage{indentfirst, tabularx}
\usepackage[table]{xcolor}
\usepackage{float, hhline, colortbl}
\usepackage{graphicx}

\newtheorem{theorem}{Theorem}[section]
\newtheorem{lemma}[theorem]{Lemma}
\newtheorem{proposition}[theorem]{Proposition}
\newtheorem{definition}[theorem]{Definition}
\newtheorem{remark}[theorem]{Remark}

\hypersetup{
	colorlinks = true,
	linkcolor = magenta,
	citecolor = blue,
}

\makeatletter
\@addtoreset{equation}{section}
\makeatother

\usepackage{color}

\title[Existence and uniqueness]{Existence and uniqueness results for a nonlinear Budiansky-Sanders shell model}
\author[T.H.Giang]{Trung Hieu Giang$^{1,2}$}
\address{Trung Hieu Giang$^{1,2}$\newline
	$^1$ Department of Mathematics, City University of Hong Kong, 83 Tat Chee Avenue, Kowloon, Hong Kong; \newline
	$^2$ Institute of Mathematics, Vietnam Academy of Science and Technology, 18 Hoang Quoc Viet, Cau Giay, Ha Noi, Vietnam.}
\email{thgiang2-c@my.cityu.edu.hk}

\begin{document}

	\begin{abstract}
		A nonlinear shell model is studied in this paper. This is a nonlinear variant of the Budiansky-Sanders linear shell model. Under some suitable assumptions on the magnitude of the applied force, we will prove the existence of a minimizer for this shell model. In addition, we will also show that our existence result can be applied to all kinds of geometries of the middle surface of the shell. We will also show that the minimizer found in this fashion is unique, provided the applied forces are small enough. Our result hence extends the one given by Destuynder in \cite{Destuyn1}.
	\end{abstract}
	
	\maketitle
	\section{Introduction}\label{sec1}
	The notation is defined in Section \ref{sec2}.
	
	During the last several decades, many studies have been devoted to developing two-dimensional shell theory. Based on different assumptions and approaches, the authors have derived various nonlinear shell models (see, for example, \cite{koi, na, ciarlet} and the references therein), such as the one by Koiter, the one by Naghdi, etc. One of these shell models is proposed by Destuynder (see \cite{Destuyn1, Destuyn2}). His shell model is a nonlinear variant of the Budiansky-Sanders linear shell model (see \cite{Budisan}), where the nonlinear contribution is only confined to the "membrane" energy term. More precisely, it states that the unknown displacement field $\boldsymbol{\eta}: \omega \to \mathbb{R}^3$ of the middle surface $S = \boldsymbol{\theta}(\overline{\omega})$ should minimize a functional
	\begin{equation}\label{s1-e1}\tag{P}
		J_{BS}({\boldsymbol{\eta}}) := \int\limits_{\omega} W_{BS}(\boldsymbol{\eta})\,\sqrt{a} d y  - \int\limits_{\omega} \boldsymbol{f} \cdot \boldsymbol{\eta} \sqrt{a}dy,
	\end{equation}
	called the total energy of the deformed shell over an appropriate set of admissible displacement fields. Here 
	$$
	W_{BS} ({\boldsymbol{\eta}}):= \dfrac{\varepsilon}{2}W_{BS}^M ({\boldsymbol{\eta}})+ \dfrac{\varepsilon^3}{6}W_{BS}^F({\boldsymbol{\eta}})
	$$
	denotes the stored energy function, $W_{BS}^M({\boldsymbol{\eta}})$ and $W_{BS}^F({\boldsymbol{\eta}})$ (which will be defined later) respectively denote the membrane and the flexural energy associated with an admissible displacement ${\boldsymbol{\eta}}$, and $\boldsymbol{f}$ denotes the density of applied force. This nonlinear shell model is supposed to be theoretically applied to the general shape of shells.
	
	The existence of minimizers for this nonlinear shell model has also been partially studied in \cite{Destuyn1}. In particular, provided some additional assumptions on the magnitude of the applied forces and the geometry of the shell, the authors obtained the following result (see \cite{Destuyn1}).
	
	\begin{theorem}\label{s1-t1}
		Provided the norms of the applied forces and those of Christoffel symbols are small enough, the nonlinear problem \eqref{s1-e1} has at least one minimizer.  
	\end{theorem}
	
	Also, in \cite{Destuyn1}, a geometrical condition is required when no restriction on the applied forces is used.
	\begin{theorem}\label{s1-t2} 
		Assuming the relations
		$$
		(\mathcal{H}) \quad \begin{cases}
			(\varphi_1(\boldsymbol{\eta}), \varphi_2(\boldsymbol{\eta})) = \boldsymbol{0}, \textrm{ and }  (\gamma_{\alpha\beta} (\boldsymbol{\eta})) + \dfrac{1}{2}(h_{\alpha}h_\beta) = \boldsymbol{0}, \\
			\textrm{ where } \boldsymbol{h} = (h_1, h_2) \in (W^{1,2}_0(\omega))^2 \textrm{ and } \boldsymbol{\eta} = \eta_i \boldsymbol{a}^i \\
			\textrm{ with } \eta_\alpha \in W^{1,2}_0 (\omega) \textrm{ and } \eta_3 \in W^{2,2}_0 (\omega),
		\end{cases}
		$$
		implies $\boldsymbol{h} = \boldsymbol{0}$, and $\boldsymbol{\eta} = \boldsymbol{0}$.
		
		Then, the nonlinear problem \eqref{s1-e1} has at least one minimizer in $\mathbb{X}_0$.
	\end{theorem}
	
	It is also worth mentioning that the uniqueness of the minimizer has been partially established in \cite{giang}, where the author considered the shell with the middle surface being a portion of a cylinder and the applied force in a particular class where elements can be made arbitrarily small or big. The existence of a solution to the boundary value problem associated with the energy functional is proved via the Inverse Function Theorem, and the uniqueness part is then followed by a comparison of the values of the energy at this solution and other displacement fields.
	
	This paper first aims to extend the result given in Theorem \ref{s1-t1}. More specifically, we will prove an existence theorem of a minimizer for the problem \eqref{s1-e1} that holds in a neighborhood of a class of special forces. These special forces can be made arbitrarily big as well as arbitrarily small. Furthermore, we will show that our existence result is valid for general shells, irrespective of their shape, and thus, our result is an improvement of Theorem \ref{s1-t1}. We will also prove that the minimizer found in this fashion is unique under the impact of small applied forces.
	
	The organization of this paper is as follows. In Section \ref{sec2}, we present notation and definitions. Section \ref{sec4} first introduces a set of special applied forces whose magnitudes can be made arbitrarily small or big. Then, we will state and prove our existence and uniqueness results.
	
	\section{Notation and definitions}\label{sec2}
	
	Throughout this paper, Greek indices and exponents range in the set $\{1,2\}$, while Latin indices and exponents range in the set $\{1,2,3\}$ (unless when they are used for indexing sequences). The summation convention with respect to repeated indexes and exponents is used.
	
	Vector fields are denoted by boldface letters. The Euclidean norm, the inner product and the vector product of two vectors $\boldsymbol{u}$ and $\boldsymbol{v}$ in $ \mathbb{R}^3$ are respectively denoted $|\boldsymbol{u}|$, $\boldsymbol{u} \cdot \boldsymbol{v}$ and $\boldsymbol{u} \wedge \boldsymbol{v}$.
	
	A domain in $\mathbb{R}^2$ is a bounded, connected, open subset $\omega \subset \mathbb{R}^2$ with a Lipschitz-continuous boundary $\gamma:= \partial\omega$, the set $\omega$ being locally on the same side of $\gamma$. A generic point in the set $\omega$ is denoted $y = (y_\alpha)$ and partial derivatives, in the classical or distributional sense, are denoted $\partial_\alpha:= \partial/\partial y_\alpha$ and $\partial_{\alpha\beta}:= \partial^2/\partial y_\alpha \partial y_\beta$.
	
	The usual norm of the Lebesgue space $L^p (\omega)$, $1 \leq p \leq \infty$, is denoted by $\| \cdot \|_p$. The usual norm of the Sobolev spaces $W^{m,p}(\omega)$, $m \in \mathbb{N}^*$, $1 \leq p \leq \infty$, is denoted by $\| \cdot \|_{m,p}$. 
	
	
	The middle surface of the reference configuration of a shell is defined by $S:= \boldsymbol{\theta}(\omega)$, where $\omega \subset \mathbb{R}^2$ is a domain and $\boldsymbol{\theta} \in \mathcal{C}^3 (\overline{\omega};\mathbb{R}^3)$ is an immersion, i.e., the two vector fields $\boldsymbol{a}_\alpha:= \partial_\alpha \boldsymbol{\theta}$ are linearly independent at every $y \in \overline{\omega}$. Next we define the vector field $\boldsymbol{a}_3: \overline{\omega} \to \mathbb{R}^3$ by 
	\begin{equation*}
		\boldsymbol{a}_3(y) := \dfrac{\boldsymbol{a}_1(y) \wedge \boldsymbol{a}_2(y)}{\left|\boldsymbol{a}_1(y) \wedge \boldsymbol{a}_2(y) \right|} \textrm{ for all } y\in\overline{\omega},
	\end{equation*} 
	which is of class $\mathcal{C}^2$ over $\overline{\omega}$. Note that $\boldsymbol{a}_3(y)$ is the positively oriented unit vector normal to the surface $S$ at the point $\boldsymbol{\theta}(y)$, while the two vectors $\boldsymbol{a}_\alpha$ form a basis in the tangent plane to $S$ at the same point. The three vectors $\boldsymbol{a}_i (y)$ constitute the \textit{covariant basis} in $\mathbb{R}^3$ at the point $\boldsymbol{\theta}(y)$, while the three vectors $\boldsymbol{a}^i (y)$ uniquely defined by the relations
	$$
	\boldsymbol{a}^i (y) \cdot \boldsymbol{a}_j (y) = \delta^i_j,
	$$
	where $\delta^i_j$ is the Kronecker symbol, constitute the \textit{contravariant basis} at the point $\boldsymbol{\theta}(y) \in S$. As a consequence, any vector field $\boldsymbol{\zeta} : \omega \to \mathbb{R}^3$ can be decomposed over either of these bases as
	$$
	\boldsymbol{\zeta} = \zeta_i \boldsymbol{a}^i = \zeta^i \boldsymbol{a}_i,
	$$
	for some functions $\xi_i: \omega \to \mathbb{R}$ and $\xi^i : \omega \to \mathbb{R}$.
	
	The area element of the surface $S$ is $\sqrt{a(y)}dy$, where
	$$
	a:=|\boldsymbol{a}_1 \wedge \boldsymbol{a}_2| \text{ in } \omega.
	$$	
	
	The covariant components $a_{\alpha\beta} \in \mathcal{C}^2 (\overline{\omega})$ and $b_{\alpha\beta} \in \mathcal{C}^1 (\overline{\omega})$ of respectively the first and second fundamental forms of $S = \boldsymbol{\theta} (\overline{\omega})$ are defined by 
	$$a_{\alpha\beta} := \boldsymbol{a}_\alpha \cdot \boldsymbol{a}_\beta \textrm{ and }  b_{\alpha\beta} :=\boldsymbol{a}_3 \cdot \partial_\alpha\boldsymbol{a}_\beta .$$
	The contravariant components of the first fundamental form are the components $a^{\alpha\beta} \in \mathcal{C}^2 (\overline{\omega})$ of the inverse matrix 
	$$
	(a^{\alpha\beta}(y)):=(a_{\alpha\beta}(y))^{-1}, \ y\in\overline{\omega}. 
	$$
	
	The mixed components $b^\alpha_\beta \in \mathcal{C}^1 (\overline{\omega})$ of the second fundamental
	form and the Christoffel symbols $\Gamma^\sigma_{\alpha\beta} \in \mathcal{C}^1 (\overline{\omega})$ of $S$ are respectively defined by
	$$b^\alpha_\beta = a^{\alpha\sigma}b_{\sigma\beta} \textrm{ in } \overline{\omega},$$
	and
	$$
	\Gamma^\sigma_{\alpha\beta} := \boldsymbol{a}^\sigma \cdot \partial_\beta \boldsymbol{a}_\alpha \textrm{ in } \overline{\omega}.
	$$
	
	A displacement field of the middle surface of the shell $S = \boldsymbol{\theta} (\overline{\omega})$ is a smooth enough vector field $\boldsymbol{\eta}: \omega \to \mathbb{R}^3$. Given an arbitrary displacement field $\boldsymbol{\eta} = \eta_i \boldsymbol{a}^i$, the functions
	$$ a_{\alpha\beta} (\boldsymbol{\theta}+\boldsymbol{\eta}) := \boldsymbol{a}_\alpha (\boldsymbol{\theta}+\boldsymbol{\eta}) \cdot \boldsymbol{a}_\beta (\boldsymbol{\theta}+\boldsymbol{\eta}), \textrm{ where } \boldsymbol{a}_\alpha (\boldsymbol{\theta}+\boldsymbol{\eta}) := \partial_\alpha (\boldsymbol{\theta}+\boldsymbol{\eta}),$$
	denote the covariant components of the first fundamental form of the deformed surface $(\boldsymbol{\theta}+\boldsymbol{\eta})(\omega)$, the functions
	$$ G_{\alpha\beta} (\boldsymbol{\theta}+\boldsymbol{\eta}) := \dfrac{1}{2} (a_{\alpha\beta}(\boldsymbol{\theta}+\boldsymbol{\eta})-a_{\alpha\beta})$$
	denote the covariant components of the change of metric tensor field associated with the deformation $\boldsymbol{\theta}+\boldsymbol{\eta}$ of $S$, the functions
	$$
	\gamma_{\alpha\beta} (\boldsymbol{\eta}) := \dfrac{1}{2}(\partial_\alpha \eta_\beta + \partial_\beta \eta_\alpha) - \Gamma^\sigma_{\alpha\beta} \eta_\sigma - b_{\alpha\beta}\eta_3
	$$
	denote the covariant components of the linearized change of metric tensor field with respect to $\boldsymbol{\eta}$, and the functions
	$$
	G^{BS}_{\alpha\beta}(\boldsymbol{\eta}) := \gamma_{\alpha\beta}(\boldsymbol{\eta}) + \dfrac{1}{2}\varphi_\alpha(\boldsymbol{\eta})\varphi_{\beta}(\boldsymbol{\eta}),
	$$
	denote the covariant components of the "modified change of metric tensor field" associated with the displacement $\boldsymbol{\eta}$, where
	$$
	\varphi_{\beta}(\boldsymbol{\eta}) := \partial_\beta \eta_3 - b^\sigma_\beta \eta_\sigma.
	$$
	Note that the original change of the metric tensor field is defined by
	$$
	G_{\alpha\beta} (\boldsymbol{\theta} + \boldsymbol{\eta}) = G^{BS}_{\alpha\beta} (\boldsymbol{\eta}) + \dfrac{1}{2} a^{\sigma\tau}(\partial_\alpha \eta_\sigma - \Gamma_{\alpha\sigma}^\gamma\eta_\gamma - b_{\alpha\sigma} \eta_3)(\partial_\beta \eta_\tau - \Gamma_{\beta\tau}^\gamma\eta_\gamma - b_{\beta\tau} \eta_3).
	$$
	If the two vectors $\boldsymbol{a}_\alpha (\boldsymbol{\theta}+ \boldsymbol{\eta}
	) $ are linearly independent, then the unit vector field
	\begin{equation*}	                  \boldsymbol{a}_3(\boldsymbol{\theta}+\boldsymbol{\eta}):=\frac{{\partial}_1(\boldsymbol{\theta}+\boldsymbol{\eta}) \wedge {\partial}_2(\boldsymbol{\theta}+\boldsymbol{\eta})}{|{\partial}_1(\boldsymbol{\theta}+\boldsymbol{\eta}) \wedge {\partial}_2(\boldsymbol{\theta}+\boldsymbol{\eta})|}
	\end{equation*}
	is well-defined and normal to the deformed surface $(\boldsymbol{\theta}+\boldsymbol{\eta} )(\omega)$. The functions
	\begin{equation*}
		b_{\alpha\beta}(\boldsymbol{\theta}+\boldsymbol{\eta}) := \boldsymbol{a}_3 (\boldsymbol{\theta}+\boldsymbol{\eta}) \cdot {\partial}_{\alpha\beta}(\boldsymbol{\theta}+\boldsymbol{\eta})
	\end{equation*}
	denote the covariant components of the second fundamental form of the deformed surface $(\boldsymbol{\theta}+\boldsymbol{\eta})(\omega)$, the functions
	$$
	R_{\alpha\beta}(\boldsymbol{\theta}+\boldsymbol{\eta}) :=  b_{\alpha\beta}(\boldsymbol{\theta}+\boldsymbol{\eta}) - b_{\alpha\beta}
	$$
	denote the covariant components of the change of curvature tensor field, the functions
	\begin{align*}
		\rho_{\alpha\beta}(\boldsymbol{\eta}) &:= \partial_{\alpha\beta} \eta_3 - \Gamma^\sigma_{\alpha\beta} \partial_\sigma \eta_3 - b^\sigma_\alpha b_{\sigma\beta} \eta_3 \\
		& \quad +b^\sigma_\alpha (\partial_\beta \eta_\sigma - \Gamma^\tau_{\beta\sigma}\eta_\tau) +b^\tau_\beta (\partial_\alpha \eta_\tau - \Gamma^\sigma_{\alpha\tau}\eta_\sigma) \\
		& \quad + (\partial_\alpha b^\tau_\beta + \Gamma^\tau_{\alpha\sigma}b^\sigma_\beta - \Gamma^\sigma_{\alpha\beta}b^\tau_\sigma)\eta_\tau
	\end{align*}
	denote the linearized change of curvature tensor field with respect to $\boldsymbol{\eta}$, and the functions
	\begin{align*}
		\rho^{BS}_{\alpha\beta} (\boldsymbol{\eta}) &:= \dfrac{1}{2}\left[ \left(\partial_\beta\varphi_{\alpha}(\boldsymbol{\eta}) - \Gamma^\tau_{\alpha\beta} \varphi_\tau (\boldsymbol{\eta})\right)+ \left(\partial_\alpha \varphi_{\beta}(\boldsymbol{\eta}) - \Gamma^\tau_{\beta\alpha} \varphi_\tau (\boldsymbol{\eta}) \right) \right] \\
		& = \partial_{\alpha\beta} \eta_3 -\dfrac{1}{2}\left( \partial_\beta b^\sigma_\alpha + \partial_{\alpha} b^\sigma_\beta \right) \eta_\sigma - \dfrac{1}{2} (b^\sigma_\alpha \partial_\beta \eta_\sigma + b^\sigma_\beta \partial_\alpha \eta_\sigma) \\
		& - \Gamma^\tau_{\alpha\beta}\partial_\tau \eta_3 + \Gamma^\tau_{\alpha\beta}\eta_\sigma b^\sigma_\tau,
	\end{align*}
	denote the "modified linearized change of curvature tensor field" with respect to $\boldsymbol{\eta}$.
	
	Let $\boldsymbol{f} = f^i \boldsymbol{a}_i$, where $f^i \in L^2 (\omega)$,  be the density of forces acting on the shell. The modified nonlinear Budiansky-Sanders shell model proposed by Destuynder \cite{Destuyn1, Destuyn2}) states that the unknown displacement $\boldsymbol{\eta}: \omega \to \mathbb{R}^3$ of the middle surface $S =\boldsymbol{\theta}(\overline{\omega})$ of the shell subjected to applied forces should minimize the functional
	\begin{equation}\label{s2-e1}
		\boldsymbol{\eta} \in \mathbb{X}_0 \to J_{BS}({\boldsymbol{\eta}}) := \int\limits_{{\omega}}          W_{BS}(\boldsymbol{\eta})\,\sqrt{a} d y  - \int\limits_\omega \boldsymbol{f} \cdot {\boldsymbol{\eta}}\sqrt{a}dy,
	\end{equation}
	defined on the space 
	\begin{equation*}
		\mathbb{X}_0 := \left\{ \boldsymbol{\eta} = \eta_i \boldsymbol{a}^i, \textrm{ } \eta_\alpha \in W^{1,2}_0 (\omega), \textrm{ } \eta_3 \in W^{2,2}_0(\omega)\right\},
	\end{equation*}
	where
	$$
	W^{1,2}_0 (\omega) = \left\{ u \in W^{1,2}(\omega), u = 0 \textrm{ on } \gamma \right\},
	$$
	$$
	W^{2,2}_0 (\omega) = \left\{ u \in W^{2,2}(\omega), u = \partial_\nu u = 0 \textrm{ on } \gamma \right\},
	$$
	and the norm on $\mathbb{X}_0$ is defined by
	$$
	\|\boldsymbol{\eta}\|_{\mathbb{X}_0} := \|\eta_1\|_{1,2} +\|\eta_2\|_{1,2} + \|\eta_3\|_{2,2} \textrm{ for all } \boldsymbol{\eta} \in \mathbb{X}_0.
	$$
	
	In \eqref{s2-e1}, 
	$$
	W_{BS}(\boldsymbol{\eta}) := \dfrac{\varepsilon}{2} W^M_{BS}(\boldsymbol{\eta}) + \dfrac{\varepsilon^3}{6}W^F_{BS}(\boldsymbol{\eta}),
	$$
	with 
	$$
	W^M_{BS}(\boldsymbol{\eta}) := a^{\alpha\beta\sigma\tau}G^{BS}_{\alpha\beta}(\boldsymbol{\eta})G_{\sigma\tau}^{BS}(\boldsymbol{\eta})
	$$
	denotes the "modified membrane energy", and
	$$
	W^F_{BS}(\boldsymbol{\eta}) := a^{\alpha\beta\sigma\tau}\rho^{BS}_{\alpha\beta}(\boldsymbol{\eta})\rho_{\sigma\tau}^{BS}(\boldsymbol{\eta})
	$$
	denotes the "modified flexural energy" associated with the displacement field $\boldsymbol{\eta}$ of the middle surface $S$, the functions
	\begin{equation*}
		a^{\alpha\beta\sigma\tau} := \dfrac{4\lambda\mu}{\lambda+2\mu} a^{\alpha\beta}a^{\sigma\tau} + 2\mu (a^{\alpha\sigma}a^{\beta\tau}+a^{\alpha\tau}a^{\beta\sigma})
	\end{equation*}
	denote the contravariant components of the two-dimensional elasticity tensor of the shell, and $\lambda \geq 0$ and $\mu > 0$ are the Lamé constants of the elastic material constituting the shell. Notice that there exists a constant $c_e = c_e (\omega, \boldsymbol{\theta}, \lambda, \mu) > 0$ such that
	\begin{equation}\label{s2-e2}
		c_e\sum\limits_{\alpha,\beta} |t_{\alpha\beta}|^2 \leq  a^{\alpha\beta\sigma\tau}(y) t_{\sigma\tau}t_{\alpha\beta},
	\end{equation}
	for all $y \in \overline{\omega}$ and all symmetric matrices $(t_{\alpha\beta})$ (see, e.g., \cite[Theorem 3.3-2]{ciarlet}).
	
	It is easy to see that the functional $J_{BS}$ is weakly lower semicontinuous. However, without additional assumptions on the geometry of the shell and/or the ones on the applied forces, it is still not known whether the functional $J_{BS}$ is coercive. Therefore, the condition on the magnitude of the applied forces given in the remaining part of this paper is crucial to obtain the existence of a minimizer.

	\section{Main results}\label{sec4}
	
	We have the following definition, which is convenient for stating our main results.
	\begin{definition}\label{s3-d1}
		We denote $\mathcal{A}$ the set consists of vector fields $\boldsymbol{f} = f^i a_i$, where $f^i \in L^2(\omega)$, satisfying that there exists a function $g \in W^{1,2}(\omega)$ such that
		$$
		f^1 = (b^1_1 + b^1_2)g, \textrm{ } f^2 = (b^2_1 + b^2_2)g,\textrm{ and }
		f^3 = \partial_1 g + \partial_2 g.
		$$
	\end{definition}
	
	\begin{proposition}\label{s3-p1}
		The set $\mathcal{A}$ is nonempty. Moreover, for each $M>0$, there exists $\boldsymbol{f} \in \mathcal{A}$ such that 
		$$
		\|f^1\|_2 + \|f^2\|_2 + \|f^3\|_2 > M.
		$$
		Furthermore, for each $\delta > 0$, there exists $\boldsymbol{f} \in \mathcal{A}$ such that 
		$$
		\|f^1\|_2 + \|f^2\|_2 + \|f^3\|_2 < \delta.
		$$
	\end{proposition}
	\begin{proof}
		Since $\boldsymbol{0} \in \mathcal{A}$, it immediately follows that the set $\mathcal{A}$ is nonempty. 
		
		Next, let $z$ be any arbitrary point in $\omega$. Since $\omega$ is open, there exists a number $n >0$ such that $\overline{B(z, \frac{1}{n})} \subset \omega$. We define the function 
		$p: \mathbb{R}^2 \to \mathbb{R}$ given by
		$$
		p(x) := \begin{cases}
			e^{1/(|x-z|^2 - 1)} & \textrm{if } |x-z| < 1, \\
			0 & \textrm{if } |x-z| \geq 1.
		\end{cases}
		$$
		Then it is easy to see that $p \in \mathcal{C}_c^\infty (\mathbb{R}^2)$, $\textrm{supp } p \subset \overline{B(z,1)}$. Therefore, the function 
		$$
		p_n(x) := p(n(x-z)+z),
		$$
		belongs to the space $\mathcal{C}_c^\infty ({\omega})$ and thus it belongs to $W^{1,2}(\omega)$. Next, it is easy to see that 
		$$
		0 < \| \partial_1 p_n + \partial_2 p_n\|_\infty < \infty.
		$$
		and thus
		$$
		0 < \| \partial_1 p_n + \partial_2 p_n\|_2 < \infty.
		$$
		
		Now we define
		$$
		g = kp_n, \textrm{ in } \omega
		$$
		where $k$ is a real number, and let 
		$$
		f^1 = (b^1_1 + b^1_2)g, \textrm{ } f^2 = (b^2_1 + b^2_2)g,\textrm{ and }
		f^3 = \partial_1 g + \partial_2 g.
		$$
		Then, we obtain the desired results by choosing an appropriately large (respectively, small) value for $k$.
	\end{proof}
	
	\begin{remark}\label{s3-r1} 
		By applying the integration by parts, it is easy to see that 
		$$\int_\omega \boldsymbol{f} \cdot \boldsymbol{\eta} \sqrt{a} dy = - \int_\omega (\varphi_1(\boldsymbol{\eta}) + \varphi_2 (\boldsymbol{\eta}))g\sqrt{a}dy$$
		for $\boldsymbol{f} \in \boldsymbol{\mathcal{A}}$. This is the motivation for introducing the set $\boldsymbol{\mathcal{A}}$, as later we can see that this set is crucial for Step 1 of the proof of Theorem \ref{s4-t1}.
	\end{remark}
	The main results of this paper are as follows.
	\begin{theorem}\label{s4-t1} 
		Assuming that the shell with middle surface $S = \boldsymbol{\theta}(\overline{\omega})$, where $\boldsymbol{\theta} \in \mathcal{C}^2 (\overline{\omega};\mathbb{R}^3)$ is an immersion (cf. Section \ref{sec2}). Let $\overline{\boldsymbol{f}} := \overline{f}^i\boldsymbol{a}_i$, where $\overline{f}^i \in L^2(\omega)$, be any arbitrary element in $\mathcal{A}$. We define the density of the applied force $\boldsymbol{f} = f^i \boldsymbol{a}_i$ by
		$$
		f^i = \overline{f}^i + h^i, \textrm{ for } i = 1, 2, 3,
		$$
		where $h^i \in L^2(\omega)$.
		
		Then, provided the norms $\|h^i\|_2$ are small enough, the problem \eqref{s1-e1} has at least one minimizer.
	\end{theorem}
	
	The following lemma plays a crucial role in the proofs of our main results.
	\begin{lemma}\label{s4-l1}
		Under the same assumptions on $S = \boldsymbol{\theta}(\overline{\omega})$ as in the statement of Theorem \ref{s4-t1}, there exists a positive constant $C_S = C_S (\omega, \boldsymbol{\theta})$ such that the following inequality holds for every $(v_1, v_2) \in (W^{1,2}_0 (\omega))^2 $
		\begin{equation}\label{s4-e1}
			\sum\limits_{\alpha}\|v_\alpha\|_{1,2}  \leq C_{S}\sum\limits_{\alpha,\beta} \|v_{\alpha | \beta} + v_{\beta | \alpha}\|_2,
		\end{equation}
		where 
		$$
		v_{\alpha|\beta} :=  \partial_\beta v_\alpha - \Gamma^\sigma_{\alpha\beta}v_\sigma.
		$$
	\end{lemma}
	\begin{proof}
		This lemma is a direct consequence of \cite[Theorem 5.5]{chen} by letting $g_{ij} = \delta_{ij}$, where $\delta_{ij}$ is the Kronecker symbol, and for this reason, the proof will be omitted here.
	\end{proof}
	
	Now we are able to give the proof of Theorem \ref{s4-t1}.
	
	\begin{proof}[Proof of Theorem \ref{s4-t1}]
		For readers' convenience, we will divide the proof into two steps. We denote $\{\boldsymbol{\eta}^n \}_{n=1}^\infty$ the minimizing sequence for the functional $J_{BS}$ in the space $\mathbb{X}_0$. Without loss of generality, we can assume that
		\begin{equation}\label{s5-e0}
			J_{BS} (\boldsymbol{\eta}^n) \leq J_{BS}(\boldsymbol{0}) = 0 \textrm{ for all } n.
		\end{equation}
		
		\noindent \textbf{Step 1:} We will prove that the sequence $\{ \sum\limits_{\alpha,\beta}\|\rho^{BS}_{\alpha\beta}(\boldsymbol{\eta}^n)\|_2\}_{n=1}^\infty$ is bounded. From the expressions of $J_{BS}$ and $f^i$ and \eqref{s2-e1}, the following inequality holds for any $\boldsymbol{\eta} \in \mathbb{X}_0$:
		\begin{align}\label{s5-e1}
			\nonumber
			J_{BS}(\boldsymbol{\eta}) \geq & \dfrac{c_e \varepsilon^3}{6} \left(\sum\limits_{\alpha,\beta}\|G^{BS}_{\alpha\beta}(\boldsymbol{\eta})\|^2_2 + \sum\limits_{\alpha,\beta}\|\rho^{BS}_{\alpha\beta}(\boldsymbol{\eta})\|^2_2\right) \\
			& - \|g\|_2 (\|\varphi_1(\boldsymbol{\eta}) \|_2 + \|\varphi_2(\boldsymbol{\eta}) \|_2) - Cl \|\boldsymbol{\eta}\|_{\mathbb{X}_0},
		\end{align}
		where $C$ is a positive constant, and
		$$
		l = \max\limits_{i = 1,2,3} \|h^i\|_2,
		$$
		and $g$ is a function in $W^{1,2}(\omega)$ such that
		$$
		\overline{f}^1 = g(b^1_1 + b^1_2), \textrm{ } \overline{f}^2 = g(b^2_1 + b^2_2),\textrm{ and } \overline{f}^3 = \partial_1 g + \partial_2 g.
		$$
		
		From Lemma \ref{s4-l1}, we deduce that
		\begin{align}\label{s5-e2}
			\nonumber
			J_{BS}(\boldsymbol{\eta}) \geq & \dfrac{c_e \varepsilon^3}{6}\left(\sum\limits_{\alpha,\beta}\|G^{BS}_{\alpha\beta}(\boldsymbol{\eta})\|^2_2 + \sum\limits_{\alpha,\beta}\|\rho^{BS}_{\alpha\beta}(\boldsymbol{\eta})\|^2_2\right) \\
			& - C_S\|g\|_2 (\sum\limits_{\alpha,\beta}\|\rho^{BS}_{\alpha\beta}(\boldsymbol{\eta})\|_2) - Cl \|\boldsymbol{\eta}\|_{\mathbb{X}_0},
		\end{align}
		with notice that 
		$$
		\rho^{BS}_{\alpha\beta}(\boldsymbol{\eta}) = \dfrac{1}{2}\left( (\varphi_\alpha (\boldsymbol{\eta}))_{|\beta} + (\varphi_\beta (\boldsymbol{\eta}))_{|\alpha} \right).
		$$
		Here $C_S$ denotes the positive constant in \eqref{s4-e1}.
		
		Next, assuming that the sequence $\{ \sum\limits_{\alpha,\beta}\|\rho^{BS}_{\alpha\beta}(\boldsymbol{\eta}^n)\|_2\}_{n=1}^\infty$ is unbounded. Then, there exists $N>0$ such that for every $n > N$, the following inequality holds
		\begin{equation}\label{s5-e3}
			\dfrac{c_e \varepsilon^3}{6}\left( \sum\limits_{\alpha,\beta}\|\rho^{BS}_{\alpha\beta}(\boldsymbol{\eta}^n)\|^2_2\right) - C_S\|g\|_2 (\sum\limits_{\alpha,\beta}\|\rho^{BS}_{\alpha\beta}(\boldsymbol{\eta}^n)\|_2) \geq \dfrac{c_e \varepsilon^3}{12}\left( \sum\limits_{\alpha,\beta}\|\rho^{BS}_{\alpha\beta}(\boldsymbol{\eta}^n)\|^2_2\right).
		\end{equation}
		Thus it follows from \eqref{s5-e2} and \eqref{s5-e3} that
		\begin{align*}
			J_{BS}(\boldsymbol{\eta}^n) \geq & \dfrac{c_e \varepsilon^3}{6} \left(\sum\limits_{\alpha,\beta}\|G^{BS}_{\alpha\beta}(\boldsymbol{\eta}^n)\|^2_2 + \dfrac{1}{2}\sum\limits_{\alpha,\beta}\|\rho^{BS}_{\alpha\beta}(\boldsymbol{\eta}^n)\|^2_2\right)  - Cl \|\boldsymbol{\eta}^n\|_{\mathbb{X}_0}.
		\end{align*}
		From this and the fact that $J_{BS}(\boldsymbol{\eta}^n) \leq 0$, we deduce that for every $n > N$, the following holds
		\begin{equation}\label{s5-e4}
			\dfrac{c_e \varepsilon^3}{6} \left(\sum\limits_{\alpha,\beta}\|G^{BS}_{\alpha\beta}(\boldsymbol{\eta}^n)\|^2_2 + \dfrac{1}{2}\sum\limits_{\alpha,\beta}\|\rho^{BS}_{\alpha\beta}(\boldsymbol{\eta}^n)\|^2_2\right)  \leq Cl \|\boldsymbol{\eta}^n\|_{\mathbb{X}_0}.
		\end{equation}
		
		By using the result in \cite[p.75]{Destuyn1} showing that the mapping
		$$
		(\zeta_1, \zeta_2, \zeta_3) \in (W^{1,2}_0(\omega))^2 \times W^{2,2}_0 (\omega) \to \sum\limits_{\alpha,\beta} \left(\| \gamma_{\alpha\beta} (\boldsymbol{\zeta}) \|_2 +\|\rho^{BS}_{\alpha\beta} (\boldsymbol{\zeta}) \|_2\right) ,
		$$
		where $\boldsymbol{\zeta} = \zeta^i \boldsymbol{a}_i$, is a norm on the space $(W^{1,2}_0(\omega))^2 \times W^{2,2}_0 (\omega)$ equivalent to the canonical one, together with the Poincaré inequality and the Sobolev embeddings, we deduce that
		\begin{align}\label{s5-e5}
			\nonumber
			\|\boldsymbol{\eta}^n\|_{\mathbb{X}_0} &\leq C' \sum\limits_{\alpha,\beta}\left( \| \gamma_{\alpha\beta} (\boldsymbol{\eta}^n) \|_2 +\|\rho^{BS}_{\alpha\beta} (\boldsymbol{\eta}^n)\|_2\right) \\
			\nonumber
			& \leq C' \sum\limits_{\alpha,\beta}\left( \| G^{BS}_{\alpha\beta} (\boldsymbol{\eta}^n) \|_2 + \|\varphi_\alpha(\boldsymbol{\eta}^n)\varphi_\beta(\boldsymbol{\eta}^n)\|_2 + \|\rho^{BS}_{\alpha\beta} (\boldsymbol{\eta}^n)\|_2\right) \\
			& \leq C' \left(\sum\limits_{\alpha,\beta}\left( \| G^{BS}_{\alpha\beta} (\boldsymbol{\eta}^n) \|_2 + \|\rho^{BS}_{\alpha\beta} (\boldsymbol{\eta}^n)\|_2\right) + \sum\limits_{\alpha}\| \varphi_\alpha (\boldsymbol{\eta}^n)\|^2_4\right).
		\end{align}
		
		Again from Lemma \ref{s4-l1}, it follows from \eqref{s5-e5} that
		\begin{align}\label{s5-e6}
			\nonumber
			\|\boldsymbol{\eta}^n\|_{\mathbb{X}_0} &\leq C' \left(\sum\limits_{\alpha,\beta}\left( \| G^{BS}_{\alpha\beta} (\boldsymbol{\eta}^n) \|_2 + \|\rho^{BS}_{\alpha\beta} (\boldsymbol{\eta}^n)\|_2\right) + \sum\limits_{\alpha}\| \varphi_\alpha (\boldsymbol{\eta}^n)\|^2_{1,2}\right) \\
			& \leq C' \left(\sum\limits_{\alpha,\beta}\left( \| G^{BS}_{\alpha\beta} (\boldsymbol{\eta}^n) \|_2 + \|\rho^{BS}_{\alpha\beta} (\boldsymbol{\eta}^n)\|_2\right) + \sum\limits_{\alpha,\beta}\|\rho^{BS}_{\alpha\beta} (\boldsymbol{\eta}^n)\|_2^2\right).
		\end{align}
		Hence provided that
		$$
		l < \dfrac{c_e \varepsilon^3}{12CC'},
		$$
		by combining \eqref{s5-e4} and \eqref{s5-e6} and let $n\to\infty$, we have a contradiction. Thus, for $l$ small enough, the sequence $\{ \sum\limits_{\alpha,\beta}\|\rho^{BS}_{\alpha\beta}(\boldsymbol{\eta}^n)\|_2\}_{n=1}^\infty$ is bounded. 
		
		\noindent\textbf{Step 2:} From Step 1, Lemma \ref{s4-l1}, the Poincar\'e inequality, and the embedding $W^{1,2}(\omega) \hookrightarrow L^4(\omega)$, it is easy to see that the sequence $\{ \sum\limits_{\alpha}\|\varphi_\alpha (\boldsymbol{\eta}^n) \|_4\}_{n=1}^\infty$ is bounded. Again, from the expressions of $J_{BS}$ and \eqref{s2-e1}, the following inequality holds for any $n$:
		\begin{align}\label{s5-e7}
			\nonumber
			0 \geq J_{BS}(\boldsymbol{\eta}^n) \geq & \dfrac{c_e \varepsilon^3}{6} \left(\sum\limits_{\alpha,\beta}\|\gamma_{\alpha\beta}(\boldsymbol{\eta}^n)\|^2_2 + \sum\limits_{\alpha,\beta}\|\rho^{BS}_{\alpha\beta}(\boldsymbol{\eta}^n)\|^2_2\right) \\
			& - C \sum\limits_{\alpha}\|\varphi_\alpha (\boldsymbol{\eta}^n) \|_4^4  - C \left(\sum\limits_{i}\|f^i\|_2 \right)\|\boldsymbol{\eta}^n\|_{\mathbb{X}_0}.
		\end{align}
		From this, the boundedness of the sequence $\{ \sum\limits_{\alpha}\|\varphi_\alpha (\boldsymbol{\eta}^n) \|_4\}_{n=1}^\infty$ and the equivalence between the canonical norm in $\mathbb{X}_0$ and the one given by 
		$$
		\sum\limits_{\alpha,\beta} \left(\| \gamma_{\alpha\beta} (\cdot) \|_2 +\|\rho^{BS}_{\alpha\beta} (\cdot) \|_2\right),
		$$
		we easily deduce that the sequence $\{\boldsymbol{\eta}^n\}_{n=1}^\infty$ is bounded in $\mathbb{X}_0$. Now, by combining with the fact that the functional $J_{BS}$ is weakly lower semicontinuous, our desired result follows. Our proof is complete.
		
	\end{proof}
	
	In the following theorem, we will study the uniqueness of the minimizer found in Theorem \ref{s4-t1}. 
	
	\begin{theorem}\label{s4-t2}
		Provided the applied force is small enough, the minimizer found in Theorem \eqref{s4-t1} is unique.
	\end{theorem}

	\begin{proof}
		Let $\boldsymbol{\eta}_{\boldsymbol{f}}$ be a minimizer of the functional $J_{BS}$ associated with the applied force $\boldsymbol{f} = f^i \boldsymbol{a}_i$. For readers' convenience, we will divide the proof into three steps.
		
		\noindent \textbf{Step 1:} It is easy to see that
		$$
		J_{BS} (\boldsymbol{\eta}_{\boldsymbol{f}}) \leq J_{BS}(\boldsymbol{0}) = 0.
		$$
		Then, by mimicking the argument in Step 1 of the proof of Theorem \ref{s4-t1}, we obtain
		\begin{align}\label{s4-t2-e1}
			\nonumber
			&\dfrac{c_e \varepsilon^3}{6}      
			\left(\sum\limits_{\alpha,\beta}\|G^{BS}_{\alpha\beta}(\boldsymbol{\eta}^n)\|^2_2 +\sum\limits_{\alpha,\beta}\|\rho^{BS}_{\alpha\beta}(\boldsymbol{\eta}^n)\|^2_2\right)  \\
			& \leq CC'l \left(\sum\limits_{\alpha,\beta}\left( \| G^{BS}_{\alpha\beta} (\boldsymbol{\eta}^n) \|_2 + \|\rho^{BS}_{\alpha\beta} (\boldsymbol{\eta}^n)\|_2\right) + \sum\limits_{\alpha,\beta}\|\rho^{BS}_{\alpha\beta} (\boldsymbol{\eta}^n)\|_2^2\right),
		\end{align}
		where 
		$$
		l := \max\limits_{i=1,2,3} \| f^i\|_2,
		$$
		$C$ and $C'$ are positive constants, and $c_e$ is the constant appearing in \eqref{s2-e2}. It follows from this inequality that
		\begin{align}\label{s4-t2-e2}
			\nonumber
			&\dfrac{c_e \varepsilon^3}{48}      
			\left(\sum\limits_{\alpha,\beta}\|G^{BS}_{\alpha\beta}(\boldsymbol{\eta}^n)\|_2 +\sum\limits_{\alpha,\beta}\|\rho^{BS}_{\alpha\beta}(\boldsymbol{\eta}^n)\|_2\right)^2 \\
			& \leq CC'l \left(\sum\limits_{\alpha,\beta}\left( \| G^{BS}_{\alpha\beta} (\boldsymbol{\eta}^n) \|_2 + \|\rho^{BS}_{\alpha\beta} (\boldsymbol{\eta}^n)\|_2\right) + \sum\limits_{\alpha,\beta}\|\rho^{BS}_{\alpha\beta} (\boldsymbol{\eta}^n)\|_2^2\right).
		\end{align}
		By letting
		$$
		l < \dfrac{c_e\varepsilon^3}{96CC'},
		$$
		we deduce from \eqref{s4-t2-e2} that
		\begin{equation}\label{s4-t2-e3}
			\dfrac{c_e \varepsilon^3}{96CC'}      
			\left(\sum\limits_{\alpha,\beta}\|G^{BS}_{\alpha\beta}(\boldsymbol{\eta}^n)\|_2 +\sum\limits_{\alpha,\beta}\|\rho^{BS}_{\alpha\beta}(\boldsymbol{\eta}^n)\|_2\right) \leq l.
		\end{equation}
		This implies that
		\begin{equation}\label{s4-t2-e4}
			\sum\limits_{\alpha,\beta}\|G^{BS}_{\alpha\beta}(\boldsymbol{\eta}^n)\|_2 \leq Ml,
		\end{equation}
		and
		\begin{equation}\label{s4-t2-e5}
			\sum\limits_{\alpha,\beta}\|\rho^{BS}_{\alpha\beta}(\boldsymbol{\eta}^n)\|_2 \leq Ml,
		\end{equation}
		where $M= \frac{96CC'}{c_e \varepsilon^3}$.
		
		\noindent \textbf{Step 2:} Let $\boldsymbol{\eta}$ be any element in $\mathbb{X}_0$ and let $\boldsymbol{\zeta} := \boldsymbol{\eta} - \boldsymbol{\eta}_{\boldsymbol{f}}$. It is obvious that $\boldsymbol{\zeta} \in \mathbb{X}_0$. Then, a series of straightforward computations show that (the detailed computations will be given in the Appendix)
		\begin{align}\label{s4-t2-e6}
			\nonumber
			J_{BS}(\boldsymbol{\eta}) = &J_{BS}(\boldsymbol{\eta}_{\boldsymbol{f}}) + \textrm{d}J_{BS}(\boldsymbol{\eta}_{\boldsymbol{f}})(\boldsymbol{\zeta}) \\
			\nonumber
			&  +\int\limits_\omega \dfrac{\varepsilon}{2}a^{\alpha\beta\sigma\tau}A_{\alpha\beta}(\boldsymbol{\eta}_{\boldsymbol{f}}, \boldsymbol{\zeta})A_{\sigma\tau}(\boldsymbol{\eta}_{\boldsymbol{f}}, \boldsymbol{\zeta})\sqrt{a}dy \\
			\nonumber
			& + \dfrac{\varepsilon}{2} \sum\limits_{\alpha,\beta,\sigma,\tau}\int\limits_\omega a^{\alpha\beta\sigma\tau}G^{BS}_{\alpha\beta}(\boldsymbol{\eta}_{\boldsymbol{f}})\varphi_{\sigma}(\boldsymbol{\zeta})\varphi_{\tau}(\boldsymbol{\zeta})\sqrt{a}dy \\
			& + \int\limits_\omega \dfrac{\varepsilon^3}{6}a^{\alpha\beta\sigma\tau}\rho^{BS}_{\alpha\beta}(\boldsymbol{\zeta})\rho^{BS}_{\sigma\tau}(\boldsymbol{\zeta})\sqrt{a}dy,
		\end{align}
		where 
		$$\aligned
		A_{\alpha\beta}(\boldsymbol{\eta}_{\boldsymbol{f}}, \boldsymbol{\zeta}) =& \gamma_{\alpha\beta}(\boldsymbol{\zeta}) + \dfrac{1}{2}\left( \varphi_\alpha (\boldsymbol{\eta}_{\boldsymbol{f}})\varphi_\beta (\boldsymbol{\zeta}) + \varphi_\alpha (\boldsymbol{\zeta})\varphi_\beta (\boldsymbol{\eta}_{\boldsymbol{f}})\right) \\
		& + \dfrac{1}{2} \varphi_\alpha (\boldsymbol{\zeta})\varphi_\beta (\boldsymbol{\zeta}).
		\endaligned$$
		
		Since $\boldsymbol{\eta}_{\boldsymbol{f}}$ is a critical point of $J_{BS}$, it follows that
		\begin{align}\label{s4-t2-e7}
			\nonumber
			J_{BS}(\boldsymbol{\eta}) = &J_{BS}(\boldsymbol{\eta}_{\boldsymbol{f}}) +\int\limits_\omega \dfrac{\varepsilon}{2}a^{\alpha\beta\sigma\tau}A_{\alpha\beta}(\boldsymbol{\eta}_{\boldsymbol{f}}, \boldsymbol{\zeta})A_{\sigma\tau}(\boldsymbol{\eta}_{\boldsymbol{f}}, \boldsymbol{\zeta})\sqrt{a}dy \\
			\nonumber
			& + \dfrac{\varepsilon}{2} \sum\limits_{\alpha,\beta,\sigma,\tau}\int\limits_\omega a^{\alpha\beta\sigma\tau}G^{BS}_{\alpha\beta}(\boldsymbol{\eta}_{\boldsymbol{f}})\varphi_{\sigma}(\boldsymbol{\zeta})\varphi_{\tau}(\boldsymbol{\zeta})\sqrt{a}dy \\
			& + \int\limits_\omega \dfrac{\varepsilon^3}{6}a^{\alpha\beta\sigma\tau}\rho^{BS}_{\alpha\beta}(\boldsymbol{\zeta})\rho^{BS}_{\sigma\tau}(\boldsymbol{\zeta})\sqrt{a}dy.
		\end{align}
		
		\noindent \textbf{Step 3:} On one hand, it follows from \eqref{s2-e2} that
		\begin{align}\label{s4-t2-e8}
			\nonumber
			&\int\limits_\omega \dfrac{\varepsilon}{2}a^{\alpha\beta\sigma\tau}A_{\alpha\beta}(\boldsymbol{\eta}_{\boldsymbol{f}}, \boldsymbol{\zeta})A_{\sigma\tau}(\boldsymbol{\eta}_{\boldsymbol{f}}, \boldsymbol{\zeta})\sqrt{a}dy + \int\limits_\omega \dfrac{\varepsilon^3}{6}a^{\alpha\beta\sigma\tau}\rho^{BS}_{\alpha\beta}(\boldsymbol{\zeta})\rho^{BS}_{\sigma\tau}(\boldsymbol{\zeta})\sqrt{a}dy \\
			&\geq \dfrac{c_e\varepsilon^3}{6} \bigg( \sum\limits_{\alpha,\beta} \| A_{\alpha\beta}(\boldsymbol{\eta}_{\boldsymbol{f}},\boldsymbol{\zeta})\|_2^2 + \sum\limits_{\alpha,\beta}\|\rho^{BS}_{\alpha\beta}(\boldsymbol{\zeta})\|_2^2\bigg).
		\end{align}
		On the other hand, thanks to \eqref{s4-t2-e4} and the H\"older inequality, we deduce that
		\begin{align}\label{s4-t2-e9}
			\nonumber
			&  \dfrac{\varepsilon}{2} \sum\limits_{\alpha,\beta,\sigma,\tau}\int\limits_\omega a^{\alpha\beta\sigma\tau}G^{BS}_{\alpha\beta}(\boldsymbol{\eta}_{\boldsymbol{f}})\varphi_{\sigma}(\boldsymbol{\zeta})\varphi_{\tau}(\boldsymbol{\zeta})\sqrt{a}dy \\
			&\geq -C_1Ml \sum\limits_{\alpha} \|\varphi_\alpha (\boldsymbol{\zeta})\|_4^2,
		\end{align}
		where $C_1$ is a positive constant depending only on $\boldsymbol{\theta}$, $\lambda$ and $\mu$.
		Again, from Lemma \ref{s4-l1}, the Poincar\'e inequality, and the embedding $W^{1,2}(\omega) \hookrightarrow L^4(\omega)$, we have
		\begin{equation}\label{s4-t2-e10}
			\sum\limits_{\alpha} \|\varphi_\alpha (\boldsymbol{\zeta})\|_4^2 \leq C_2\sum\limits_{\alpha,\beta}\|\rho^{BS}_{\alpha\beta}(\boldsymbol{\zeta})\|_2^2,
		\end{equation}
		where $C_2$ is a positive constant. Now, by combining \eqref{s4-t2-e7} - \eqref{s4-t2-e10}, we obtain
		\begin{align}\label{s4-t2-e11}
			\nonumber
			J_{BS}(\boldsymbol{\eta}) &\geq J_{BS}(\boldsymbol{\eta}_{\boldsymbol{f}}) \\
			& + \bigg(\dfrac{c_e\varepsilon^3}{6} - C_1C_2Ml\bigg) \bigg( \sum\limits_{\alpha,\beta} \| A_{\alpha\beta}(\boldsymbol{\eta}_{\boldsymbol{f}},\boldsymbol{\zeta})\|_2^2 + \sum\limits_{\alpha,\beta}\|\rho^{BS}_{\alpha\beta}(\boldsymbol{\zeta})\|_2^2\bigg).
		\end{align}
		
		Next, letting 
		$$
		l < \dfrac{c_e\varepsilon^3}{6C_1C_2M},
		$$
		and assuming that 
		$$
		J_{BS}(\boldsymbol{\eta}) = J_{BS}(\boldsymbol{\eta}_{\boldsymbol{f}}).
		$$
		From \eqref{s4-t2-e11} we deduce that
		$$
		\| A_{\alpha\beta}(\boldsymbol{\eta}_{\boldsymbol{f}},\boldsymbol{\zeta})\|_2 =  \|\rho^{BS}_{\alpha\beta}(\boldsymbol{\zeta})\|_2 = 0, \textrm{ for all } \alpha,\beta.
		$$
		This implies that
		$$
		\| \gamma_{\alpha\beta}(\boldsymbol{\zeta})\|_2 =  \|\rho^{BS}_{\alpha\beta}(\boldsymbol{\zeta})\|_2 = 0, \textrm{ for all } \alpha,\beta.
		$$
		Then we infer from the result in \cite[p.75]{Destuyn1} showing that the mapping
		$$
		(\zeta_1, \zeta_2, \zeta_3) \in (W^{1,2}_0(\omega))^2 \times W^{2,2}_0 (\omega) \to \sum\limits_{\alpha,\beta} \left(\| \gamma_{\alpha\beta} (\boldsymbol{\zeta}) \|_2 +\|\rho^{BS}_{\alpha\beta} (\boldsymbol{\zeta}) \|_2\right) ,
		$$
		where $\boldsymbol{\zeta} = \zeta^i \boldsymbol{a}_i$, is a norm on the space $(W^{1,2}_0(\omega))^2 \times W^{2,2}_0 (\omega)$ equivalent to the canonical one, we obtain
		$$
		\boldsymbol{\zeta} = \boldsymbol{0}.
		$$
		This completes our proof.
	\end{proof}
	
	\section*{Appendix: Detailed proof of equation \eqref{s4-t2-e6}}
	In the following, we will give the detailed proof of \eqref{s4-t2-e6}. First, recall that the functions $G^{BS}_{\alpha\beta}(\boldsymbol{\eta})$ are given by
	$$G^{BS}_{\alpha\beta}(\boldsymbol{\eta}) = \gamma_{\alpha\beta} (\boldsymbol{\eta}) + \dfrac{1}{2} \varphi_\alpha (\boldsymbol{\eta})\varphi_\beta(\boldsymbol{\eta}), \textrm{ for every } \boldsymbol{\eta}.$$ 
	Then, for every $\boldsymbol{\zeta} \in (W^{1,2}_0(\omega))^2 \times W^{2,2}_0 (\omega)$, we have
	\begin{align}\label{ap1}
		\nonumber
		G^{BS}_{\alpha\beta}(\boldsymbol{\eta_f} + \boldsymbol{\zeta}) & = G^{BS}_{\alpha\beta}(\boldsymbol{\eta_f}) + \gamma_{\alpha\beta}(\boldsymbol{\zeta}) \\
		\nonumber
		& + \dfrac{1}{2} (\varphi_\alpha (\boldsymbol{\zeta})\varphi_\beta(\boldsymbol{\eta_f}) + \varphi_\alpha(\boldsymbol{\eta_f})\varphi_\beta (\boldsymbol{\zeta})) \\
		\nonumber
		& + \dfrac{1}{2}\varphi_\alpha(\boldsymbol{\zeta})\varphi_\beta(\boldsymbol{\zeta}) \\
		\tag{A1}
		& = G^{BS}_{\alpha\beta}(\boldsymbol{\eta_f}) + \textrm{d}G^{BS}_{\alpha\beta}(\boldsymbol{\eta_f})(\boldsymbol{\zeta}) + \dfrac{1}{2}\varphi_\alpha(\boldsymbol{\zeta})\varphi_\beta(\boldsymbol{\zeta}),
	\end{align}
	where 
	$$ \textrm{d}G^{BS}_{\alpha\beta}(\boldsymbol{\eta_f})(\boldsymbol{\zeta}) := \gamma_{\alpha\beta}(\boldsymbol{\zeta}) + \dfrac{1}{2} (\varphi_\alpha (\boldsymbol{\zeta})\varphi_\beta(\boldsymbol{\eta_f}) + \varphi_\alpha(\boldsymbol{\eta_f})\varphi_\beta (\boldsymbol{\zeta})).$$
	Then, it follows from \eqref{ap1} that
	\begin{align}\label{ap2}
		\nonumber
		G^{BS}_{\alpha\beta}(\boldsymbol{\eta_f} + \boldsymbol{\zeta}) G^{BS}_{\sigma\tau}(\boldsymbol{\eta_f}+\boldsymbol{\zeta}) & = G^{BS}_{\alpha\beta}(\boldsymbol{\eta_f})G^{BS}_{\sigma\tau}(\boldsymbol{\eta_f}) \\
		\nonumber
		& + G^{BS}_{\alpha\beta}(\boldsymbol{\eta_f}) \textrm{d}G^{BS}_{\sigma\tau}(\boldsymbol{\eta_f})(\boldsymbol{\zeta}) +  G^{BS}_{\sigma\tau}(\boldsymbol{\eta_f})\textrm{d}G^{BS}_{\alpha\beta}(\boldsymbol{\eta_f}) (\boldsymbol{\zeta}) \\
		\nonumber
		& + \dfrac{1}{2} G^{BS}_{\alpha\beta}(\boldsymbol{\eta_f}) \varphi_{\sigma}(\boldsymbol{\zeta})\varphi_\tau (\boldsymbol{\zeta}) + \dfrac{1}{2} G^{BS}_{\sigma\tau}(\boldsymbol{\eta_f}) \varphi_{\alpha}(\boldsymbol{\zeta})\varphi_\beta (\boldsymbol{\zeta})\\
		\tag{A2}
		&+ A_{\alpha\beta}(\boldsymbol{\eta_f}, \boldsymbol{\zeta})A_{\sigma\tau}(\boldsymbol{\eta_f}, \boldsymbol{\zeta}),
	\end{align}
	where $A_{\alpha\beta}(\boldsymbol{\eta_f}, \boldsymbol{\zeta})$ is given in Step 2 of the proof of Theorem \ref{s4-t2}.
	
	Thanks to the symmetry of the elasticity tensor $(a^{\alpha\beta\sigma\tau})$, we easily deduce that
	\begin{align}\label{ap3}
		\nonumber
		& \dfrac{1}{2}\sum\limits_{\alpha,\beta,\sigma,\tau} a^{\alpha\beta\sigma\tau} \bigg(G^{BS}_{\alpha\beta}(\boldsymbol{\eta_f}) \varphi_{\sigma}(\boldsymbol{\zeta})\varphi_\tau (\boldsymbol{\zeta}) + G^{BS}_{\sigma\tau}(\boldsymbol{\eta_f}) \varphi_{\alpha}(\boldsymbol{\zeta})\varphi_\beta (\boldsymbol{\zeta}) \bigg) \\
		\tag{A3}
		& = \sum\limits_{\alpha,\beta,\sigma,\tau} a^{\alpha\beta\sigma\tau} G^{BS}_{\alpha\beta}(\boldsymbol{\eta_f}) \varphi_{\sigma}(\boldsymbol{\zeta})\varphi_\tau (\boldsymbol{\zeta}).
	\end{align}
	
	Also, notice that
	\begin{align}\label{ap4}
		\nonumber
		\rho^{BS}_{\alpha\beta}(\boldsymbol{\eta_f} + \boldsymbol{\zeta})\rho^{BS}_{\sigma\tau}(\boldsymbol{\eta_f} + \boldsymbol{\zeta})  = & \rho^{BS}_{\alpha\beta}(\boldsymbol{\eta_f})\rho^{BS}_{\sigma\tau}(\boldsymbol{\eta_f}) +  \rho^{BS}_{\alpha\beta}(\boldsymbol{\eta_f})\rho^{BS}_{\sigma\tau}(\boldsymbol{\zeta}) \\
		\tag{A4}
		& + \rho^{BS}_{\alpha\beta}(\boldsymbol{\zeta})\rho^{BS}_{\sigma\tau}(\boldsymbol{\eta_f}) + \rho^{BS}_{\alpha\beta}(\boldsymbol{\zeta})\rho^{BS}_{\sigma\tau}(\boldsymbol{\zeta}),
	\end{align}
	and
	\begin{align}\label{ap5}
		\nonumber
		\textrm{d}J_{BS} (\boldsymbol{\eta_f})(\boldsymbol{\zeta}) = & \dfrac{\varepsilon}{2}\sum\limits_{\alpha,\beta,\sigma,\tau}\int\limits_\omega a^{\alpha\beta\sigma\tau} \bigg[G^{BS}_{\alpha\beta}(\boldsymbol{\eta_f}) \textrm{d}G^{BS}_{\sigma\tau}(\boldsymbol{\eta_f})(\boldsymbol{\zeta}) \\
		\nonumber
		& \quad \quad +  G^{BS}_{\sigma\tau}(\boldsymbol{\eta_f})\textrm{d}G^{BS}_{\alpha\beta}(\boldsymbol{\eta_f}) (\boldsymbol{\zeta}) \bigg]\sqrt{a}dy \\
		\nonumber 
		& + \dfrac{\varepsilon^3}{6}\sum\limits_{\alpha,\beta,\sigma,\tau}\int\limits_\omega a^{\alpha\beta\sigma\tau} \bigg[ \rho^{BS}_{\alpha\beta}(\boldsymbol{\eta_f})\rho^{BS}_{\sigma\tau}(\boldsymbol{\zeta}) \\
		\tag{A5}
		& \quad \quad + \rho^{BS}_{\alpha\beta}(\boldsymbol{\zeta})\rho^{BS}_{\sigma\tau}(\boldsymbol{\eta_f}) \bigg]\sqrt{a}dy.
	\end{align}
	Then, by combining \eqref{ap2}, \eqref{ap3}, \eqref{ap4} and \eqref{ap5} we obtain equation \eqref{s4-t2-e6}.
	
	\section*{Acknowledgments}  The author would like to thank Professor Cristinel Mardare for his encouragement and advice.
	
	\section*{Disclosure statement}
	The author has no conflicts of interest to declare that are relevant to the content of this article.
	
	\bibliographystyle{acm}

\begin{thebibliography}{10}
		
		\bibitem{Destuyn1}
		{\sc Destuynder, P.}
		\newblock An existence theorem for a nonlinear shell model in large displacements analysis.
		\newblock {\em Math. Meth. Appl. Sci.}, (1983),
		5: 68-83.
		
		\bibitem{Destuyn2}
		{\sc Destuynder, P.}
		\newblock On nonlinear membrane theory.
		\newblock {\em  Comput. Methods Appl. Mech. Engrg. 32}, (1982), no. 1–3, 377–399.
		
		\bibitem{Budisan}
		{\sc Budiansky, P. and Sanders, J.L.}
		\newblock On the "best" first order linear shell theory.
		\newblock {\em Prog. in Appl. Mech. W. Prager. Anniversary, Volume.} New York: Mc. Millan 1967; 129–140.
		
		\bibitem{koi}
		{\sc Koiter, W.T.}
		\newblock On foundations of linear theory of thin elastic shells.
		\newblock {\em Proc. Kon. Ned. Akad. Wetensch.} B73 (1970) 169-195.
		
		
		\bibitem{na}
		{\sc Naghdi, P.M.} 
		\newblock{\em Foundations of Elastic Shell Theory.}
		\newblock (I. N. Sneddon \& R. Hill, Editors). Progress in Solid Mechanics (Vol. 4, pp. 1–90). North-Holland, Amsterdam, 1963.
		
		
		\bibitem{ciarlet}
		{\sc Ciarlet, P.G.}
		{\it Mathematical Elasticity, Volume III: Theory of Shells}. North-Holland, Amsterdam, 2000.
		
		\bibitem{chen}
		{\sc Chen, W. and Jost, J.A.}
		\newblock Riemannian version of Korn's inequality. 
		\newblock {\em Calc. Var.} 14, 517–530 (2002).
		
		\bibitem{giang}
		{\sc Giang, T.H.}
		\newblock Existence and uniqueness of minimizing solution for a nonlinear clamped cylindrical shell model.
		\newblock {\em J. Elliptic Parabol. Equ}. 10:979-995 (2024).
		
		\bibitem{acha}
		{\sc Acharya, A.}
		\newblock A nonlinear generalization of the Koiter-Sanders-Budiansky bending strain measure.
		\newblock {\em Int. J. Solids Struct.} 37, 5517–5528 (2000).
		
		\bibitem{ghi}
		{\sc Ghiba, L.D., Lewintan, P., Sky, A. and Neff, P.}
		\newblock An essay on deformation measures in isotropic thin shell theories. Bending versus curvature.
		\newblock {\em Math. Mech. Solids.} (2024)
		
	\end{thebibliography}

\end{document}